\documentclass[11pt]{article}
\usepackage{amsmath,amsthm,amsfonts,amssymb,bm,mathrsfs,wasysym}
\usepackage{epsfig}
\usepackage[usenames]{color}
\usepackage{verbatim}
\usepackage{hyperref}
\usepackage{multicol}
\usepackage{comment}
\usepackage{float}
\usepackage{graphicx}
\usepackage[utf8]{inputenc}

\def\bs{\backslash}
\def\cS{\mathcal{S}}
\def\cI{\mathcal{I}}
\def\bs{\backslash}

\parskip \medskipamount
\newtheorem{theorem}{Theorem}[section]

\numberwithin{equation}{section}
\newtheorem{lemma}[theorem]{Lemma}
\newtheorem{proposition}[theorem]{Proposition}
\newtheorem{corollary}[theorem]{Corollary}

\newtheorem{remark}[theorem]{Remark}

\numberwithin{equation}{section}

\def\N{\mathbb{N}}
\def\Z{\mathbb{Z}}

\def\R{\mathbb{R}}
\def\C{\mathcal{C}}
\def\V{\mathcal{V}}

\def\A{\mathcal{A}}
\def\cR{\mathcal{R}}
\def\cO{\mathcal{O}}
\def\CC{\mathcal{C}}

\def\cF{\mathcal{F}}
\def\cE{\mathcal{E}}

\def\bE{\mathbb{E}}
\def\bP{\mathbb{P}}

\renewcommand{\phi}{\varphi}
\renewcommand{\epsilon}{\varepsilon}

\def\K{\mathcal{K}}
\def\F{\mathcal{F}}
\def\G{\mathcal{G}}
\def\RR{\mathcal{R}}
\def\C{{\mathcal C}}

\allowdisplaybreaks

\newcommand{\1}{{\text{\Large $\mathfrak 1$}}}

\newcommand{\var}{\operatorname{var}}

\renewcommand{\emptyset}{\varnothing}

\def\reff#1{(\ref{#1})}

\newcommand{\tn}{|\kern-.1em|\kern-0.1em|}

\newcommand{\cp}{\mathrm{Cap}}

\newcommand{\cc}[1]{\mathrm{Cap}\left(#1\right)}

\def\bP{\mathbb{P}}

\def\tilde{\widetilde}

\begin{document}
\title{\bf Moderate deviations for the range of a transient random walk. II}

\author{Amine Asselah \thanks{
LAMA, Univ Paris Est Creteil, Univ Gustave Eiffel, UPEM, CNRS, F-94010, Cr\'eteil, France; amine.asselah@u-pec.fr} \and
Bruno Schapira\thanks{Aix-Marseille Universit\'e, CNRS, Centrale Marseille, I2M, UMR 7373, 13453 Marseille, France;  bruno.schapira@univ-amu.fr} 
}
\date{}
\maketitle
\begin{abstract}
We obtain sharp upper and lower bounds for 
the moderate deviations of the volume of the range of a 
random walk in dimension five and larger. 
Our results encompass two regimes: 
a Gaussian regime for small deviations, and a 
stretched exponential regime for larger deviations.
In the latter regime, we show that 
conditioned on the moderate deviations event,
the walk folds a small part of its range in a ball-like subset.  
Also, we provide new path properties, in dimension three as well.
Besides the key role Newtonian capacity plays in this study,
we introduce two original ideas, of general interest, 
which strengthen the approach developed in \cite{AS}.
\\

MSC 2010 \emph{subject classifications.} Primary 60F05, 60G50.
\end{abstract}

\section{Introduction}\label{sec-intro}
\paragraph{General overview}
In this paper, we consider a classical problem in probability theory:
the moderate deviations for the range of a simple
random walk. Curiously enough, moderate deviations estimates in dimension
five and larger are missing. Nothing is known neither on the path
properties of a walk conditioned on {\it moderately squeezing} its range.
This motivates the present work. We extend the large deviations 
analysis of \cite{AS} up to the Gaussian regime. Our approach 
also sheds light on the path properties in dimension three.

We denote by $\{S_n\}_{n\in\N}$ the simple random walk on $\Z^d$, $d\ge 3$. 
Its range is the process $n\mapsto \cR_n:=
\{S_0,\dots,S_n\}$ which records the visited 
sites as time evolves, and we denote by $|\cR_n|$ the number
of sites in $\cR_n$. In 1951, Dvoretzky and
Erd\"os \cite{DE} prove a strong law of large numbers
in dimension three and larger, with an almost sure limit
\begin{equation}\label{SLLN}
\frac 1n |\cR_n| \longrightarrow \gamma_d,  \qquad
\text{with}\qquad \gamma_d>0.
\end{equation}
They express the range as a sum over newly visited sites,
and obtain $\gamma_d$ as the probability of not 
returning to 0. 

The deviation of $|\cR_n|$ below its mean has
a long history. Kac and Luttinger in 1973 \cite{KL} discuss
the problem of evaluating the expectation of $\exp(- |\cR_n|)$ in their
study of Bose-Einstein condensation in the presence of impurities.
The first order asymptotics of $\bE[\exp(- |\cR_n|)]$
are derived a few years later
by Donsker and Varadhan in two celebrated works \cite{DV-75} and
\cite{DV-79}. The polymer measure obtained by tilting the random
walk law by $\exp(- |\cR_n|)$ is shown by Bolthausen \cite{B-94} to correspond to
confining the random walk in a ball of radius $n^{1/(d+2)}$ and
to leave no hole in the covered region, in $d=2$. Sznitman \cite{Sz91} shows in $d=2$ as well the confinement of the Wiener sausage, the continuous counterpart of the walk, 
and Povel \cite{Po} extends the proof of confinement to all dimensions $d\ge 3$. Recently a more
precise statement 
for the range has been obtained in
$d\ge 3$ independently by Ding, Fukushima, Sun and Xu \cite{DKSX} and Berestycki and Cerf 
\cite{BC}. The most relevant regime for us is that of Large
Deviations. Many studies of either the Wiener sausage or the random
walk culminate with a large deviation principle for the Wiener sausage 
in dimensions $d\ge 3$, proved 
by van den Berg, den Hollander and Bolthausen \cite{BBH01}. The result which was adapted by Phetpradap \cite{Phet} to the random walk setting reads as follows:  
for any $b>0$, if $\bP$ is the law of the walk,
\begin{equation}\label{BBH01}
\lim_{n\to\infty} \frac{1}{n^{(d-2)/d}} \cdot
\log \bP\big( |\cR_n|-\bE[|\cR_n|]< -b n\big)=-\cI^d(b),
\end{equation}
where $\cI^d(b)>0$ for any $b>0$, 
and $\cI^d$ has a variational expression in terms
of the occupation density of the Brownian motion. This suggests
that the walk is localized a time of order $n$ in a ball-like 
region of volume of order $n$. 
This suggests also an occupation density of order 1,
and the presence of {\it holes} of side-length 1.
This picture has been
popularized under the name {\it Swiss cheese} \cite{BBH01}, 
and is conjecturally linked with the model of Random Interlacements, see \cite{Sz-19}.

\paragraph{Main results.}
We first estimate the probability of deviation in $d\ge 5$. 
\begin{theorem}\label{theo.d5}
Assume $d\ge 5$. There exist positive 
constants $\epsilon$, $\underline \kappa_d$ and $\overline \kappa_d$, 
such that for any $n\ge 2$, and $ n^{\frac{d}{d+2}}\cdot (\log n)^{\frac{2d}{d^2-4}} \le \zeta \le \epsilon n$, 
\begin{equation}\label{bounds-d5}
\exp\left(-\underline \kappa_d\, \zeta^{1-\frac 2d} \right) 
\le \bP\left(|\RR_n|-\bE[|\RR_n|] \le -\zeta\right) 
\le \exp\left(-\overline \kappa_d \, \zeta^{1-\frac 2d} \right).
\end{equation}
\end{theorem}
\begin{remark}\emph{We believe that the logarithm that 
appears in the hypotheses on $\zeta$ for \eqref{bounds-d5} 
are artefacts of our proofs. 
}
\end{remark} 

Recall that a Central Limit Theorem in dimension $5$ 
and higher was proved by Jain and Orey \cite{JO} for $|\cR_n|$. 
In particular,
they show that the variance of $|\cR_n|$ divided by $n$ converges to
$\sigma^2>0$. We extend their result into a
Moderate Deviation principle.
\begin{theorem}\label{theo-gauss}
Assume $d\ge 5$.
For  any sequence $\{\zeta_n\}_{n\ge 0}$, satisfying $\lim_{n\to \infty} \zeta_n/\sqrt n = \infty$, and $\lim_{n\to \infty} 
\zeta_n (\log n)^{\frac{d-2}{d+2}} /  n^{\frac{d}{d+2}} = 0$, we have
\begin{equation}\label{limit-gauss}
\lim_{n\to \infty} \frac{n}{\zeta_n^2} \cdot 
\log \bP\left(\pm (|\RR_n|-\bE[|\RR_n|])  > \zeta_n\right) 
=   -\frac{1}{2\sigma^2}. 
\end{equation}
\end{theorem}

\begin{remark}\label{rem-d3}
\emph{
In $d=3$, \cite[Theorem 8.5.3]{Chen} 
proves a moderate deviations principle up to the Gaussian regime: 
for $\{\zeta_n\}_{n\in \N}$ with 
$\lim_{n\to \infty} \zeta_n/(\sqrt n\cdot \log^{3/4}(n))=\infty$, 
we have
\begin{equation}\label{Xia-1}
\lim_{n\to\infty} \frac{n^{1/3}}{\zeta_n^{2/3}} \cdot
\log \bP\big( |\cR_n|-\bE[|\cR_n|]< -\zeta_n\big)=-\cI_1,
\end{equation}
with an explicit positive constant $\cI_1$. On the other hand, 
if $\lim_{n\to \infty} \zeta_n/( \sqrt n\cdot \log^{3/4}(n))=0$, 
\cite[Theorem 8.5.2]{Chen} shows that the deviations are Gaussian: 
\begin{equation}\label{Xia-2}
\lim_{n\to\infty} \frac{n\cdot \log n}{\zeta_n^{2}} \cdot
\log \bP\big( \pm (|\cR_n|-\bE[|\cR_n|])>\zeta_n\big)=-\cI_2,
\end{equation}
with an explicit positive constant $\cI_2$.
Note that the variance of $|\cR_n|$ is of order 
$n\log n$ in $d=3$, by \cite{JP}. 
Moderate deviations for the range in $d=2$
are analyzed in \cite{BCR} and \cite{Chen}. }

\emph{Our simple method does
provide rougher estimates that we present since our pathwise analysis
only depends on them. There exist positive constants $\epsilon$, 
$\underline \kappa$ and $\overline \kappa$, 
such that for any $n\ge 2$, 
and $n^{5/7}\cdot \log n \le \zeta \le \epsilon n$,
\begin{equation}\label{theo.d3}
\exp\left(-\underline \kappa\, (\zeta^2/n)^{1/3}  \right) 
\le \bP\left(|\RR_n|-\bE[|\RR_n|] \le -\zeta\right) 
\le \exp\left(-\overline \kappa \, (\zeta^2/n)^{1/3} \right).
\end{equation}
For an explanation of why our techniques cannot reach the Gaussian 
regime, and more precisely for the reason of this exponent $5/7$, 
see Remark \ref{rem.5/7}. }
\end{remark}

We provide now a description of a 
typical trajectory conditioned on the deviations. 
First, recall that the capacity of a set $A\subseteq \Z^d$ is defined by 
\begin{equation}\label{cap.def}
\cp(A):=\sum_{x\in A} \bP_x(S_n\not\in A,\ \forall n\ge 1).
\end{equation}
Also, given $\Lambda\subseteq \Z^d$, and $n\ge 0$, 
let $\ell_n(\Lambda)$ be the time spent in $\Lambda$ before time $n$. 
Finally, we introduce the following notation: 
for $r>0$, and $x\in \Z^d$, set  
$$
Q(x,r):=[x-\frac r2,x+\frac r2)^d\cap \Z^d,
$$
and for $\rho \in (0,1]$, and $r,n$ positive integers, we let 
\begin{equation}\label{def-CV}
\C_n(r,\rho):= \{x\in r\Z^d\, :\, \ell_n(Q(x,r))\ge \rho r^d\},
\qquad\text{and }\qquad
\V_n(r,\rho):=\bigcup_{x\in \C_n(r,\rho)} Q(x,r).
\end{equation}
In words, $\V_n(r,\rho)$ is the region of space where the walk realizes 
an {\it occupation density} above $\rho$ on a {\it space-scale} $r$.
Define also
for a sequence of values of deviation $\{\zeta_n\}_{n\in \N}$, 
\begin{eqnarray*}
\rho_{\textrm{typ}} := \left\{ 
\begin{array}{ll}
\zeta_n/n & \text{if }d=3\\
1 & \text{if }d\ge 5 
\end{array}
\right. 
\quad 
\tau_{\textrm{typ}}:= \left\{ 
\begin{array}{ll}
n & \text{if }d=3\\
\zeta_n & \text{if }d\ge 5
\end{array}
\right. 
\quad \text{and}\quad  
\chi_d:= \left\{ 
\begin{array}{ll}
5/7 & \text{if }d=3\\
\frac d{d+2} & \text{if }d\ge 5. 
\end{array}
\right. 
\end{eqnarray*}
\begin{theorem}\label{theo-path}
Assume $d=3$, or $d\ge 5$. There are positive constants $\alpha$, 
$\beta$, $\epsilon$, and $C_0$, such that for any 
sequence  $\{\zeta_n\}_{n\in \N}$ satisfying 
\begin{equation}\label{cond-pathwise}
n^{\chi_d}\cdot \log n \le \zeta_n\le \epsilon n,
\end{equation}
defining $\{r_n\}_{n\in \N}$ by 
\begin{equation}\label{def.rn}
r_n^{d-2}\rho_{\textrm{typ}}= C_0\log n, 
\end{equation}
one has 
\begin{equation}\label{result-pathwise}
\lim_{n\to \infty} \bP\left(
\ell_n(\V_n(r_n,\beta\rho_{\textrm{typ}}))
\ge \alpha\, \tau_{\textrm{typ}} \mid |\RR_n|-\bE[|\RR_n|]
\le -\zeta_n\right) = 1.
\end{equation}
Moreover, for some constant $A>0$, 
\begin{equation}\label{result-capacity}
\lim_{n\to \infty} \bP\left(
\cp(\V_n(r_n,\beta\rho_{\textrm{typ}})) 
\le A |\V_n(r_n,\beta\rho_{\textrm{typ}})|^{1-\frac 2d} 
\mid |\RR_n|-\bE[|\RR_n|] \le -\zeta_n\right) = 1. 
\end{equation}
\end{theorem}
Theorem~\ref{theo-path} provides some information on 
the density the random walk has to realize in order to achieve
the deviation. We obtain that $\V_n(r_n,\rho_n)$ is typically ball-like, 
in the sense that its capacity is of the order of its volume 
to the power $1-2/d$, as it is the case for Euclidean balls.

This result also reflects a very different phenomenology in
dimension three and in dimensions larger than four.
It is easier for this purpose
to switch to the language of polymer. There, $S_i$ is the position
of the $i$-th monomer, and $n$ is the total length. Squeezing
the range of the walk is now called {\it folding the polymer}. 
\begin{itemize}
\item In dimension three, the walk spends a fraction of its
total time, independently of $\zeta_n$, in a region of
volume of order $n^2/\zeta_n$.  
Note in particular that the volume of the confinement
region is a decreasing function of $\zeta_n$.

\item In dimensions $d\ge 5$, the density of the
confinement region is of order 1, no matter what
$\zeta_n$ is. Also, note that the region where the walk
is confined has volume of order $\zeta_n$, which increases
with $\zeta_n$. Thus, a very tiny fraction of the monomers folds.

\item 
Note also that the scale $r_n$ measures the distance at which
we can probe density, and the condition \reff{cond-pathwise}
is the smallest one can hope, up to constant. Indeed, for any constant $c$ 
small enough, the simple random walk typically spends a time larger than  
$cr^2\log n$ in order $n/(r^2\log n)$ cubes of side-length $r$ (recall that the probability to spend a time $t$ in a cube of side $r$ is of order $\exp(-t/r^2)$, ignoring constants in the exponential). 
So a density $\rho = cr^{2-d}\log n$ is typical in this respect, and \reff{cond-pathwise} simply requires 
$\rho_{\textrm{typ}}$ being larger than this typical density without
constraint.

\item
Let us quote \cite{BBH01}: {\it " The central limit theorem is controlled by
the local fluctuations, while the [..]  deviations are controlled by
the global fluctuations"}. Thus, we establish folding which is somehow
a signature of global fluctuations. Note that in a recent preprint,
on a related problem, Sznitman \cite{Sz-19} performs a {\it decomposition of
random interlacements into "wavelet" and "undertow" components,
which carry "local" versus "longer range" information}. This is
also present in the approach of \cite{AS}, and is here expressed for instance in Corollary \ref{cor-trans} and its proof (see notably the standard decomposition 
\eqref{dec.fluct}).

\item In dimension $4$, by following our approach,
we could show the following. For any $n\ge 2$, and $n^{2/3}\cdot 
\log n \le \zeta \le \epsilon n$, 
\begin{equation*}
\exp\left(-\underline \kappa\, \sqrt{\zeta} \right) 
\le \bP\left(|\RR_n|-\bE[|\RR_n|] \le -\zeta\right) 
\le \exp\left(-\overline \kappa \, 
\frac{\sqrt{\zeta}}{\log (n/\zeta)} \right),
\end{equation*}
with $\underline \kappa$ and $\overline \kappa$ two positive constants. 
It remains open to remove the $\log$, 
and most interestingly to give some information on the optimal scenario. 
\end{itemize}

\paragraph{New Tools.}
To put in perspective the present ideas, let us
describe the approach of \cite{AS}. 
The key object with which we measure folding is a set of
monomers $\K_n(r,\rho)$, parametrized by a space-scale $r$ and
a density $\rho$, which collects the monomers' labels around which,
on a scale $r$, the density is above $\rho$. In other words,
\begin{equation}\label{def-Kn}
\K_n(r,\rho):=\{k\le n:\ \ell_n(Q(S_k,r))\ge \rho\cdot r^d\}.
\end{equation}
Typically, a random walk has most sites surrounded by a density of order
$r^{2-d}$ in a cube of side-length $r$. 
Thus, folding occurs when observing an abnormally
large value of $|\K_n(r,\rho)|$, with $\rho r^{d-2}$ large. 
Our main tool is the following
deviation result which weakens the assumption of
Lemma 2.4 from \cite{AS}, and is key in tackling the
moderate deviations up to the Gaussian regime.
\begin{theorem}
\label{theo-Kn}
There exist positive constants $C_0$ and $\kappa$, 
such that for any $\rho>0$, $r\ge 1$, and $n\ge 2$, satisfying  
\begin{equation}\label{hyp.r}
\rho\cdot r^{d-2} \ge C_0\log n,
\end{equation} 
one has for any $L\ge 1$, 
\begin{equation*}
\mathbb P\big(|\K_n(r,\rho)|\ge L\big)\, 
\le \, C \exp\left(-\kappa\, \rho^{\frac 2d}\, L^{1-\frac 2d} \right).
\end{equation*}
\end{theorem}
Note that as mentioned earlier the condition \eqref{hyp.r} is optimal, up to the constant. 

The heart of our approach relies on estimating
the probability a random walk realizes a certain density
in a certain number of balls of a fixed radius. In
the regime of moderate deviation, the number of possible
balls is often larger than the reciprocal of the probability
of filling a fixed configuration of balls. 
We reduce the complexity of the set
of centers with the following idea (see Lemma \ref{lem-cap} 
below for a more general and formal statement). 

\begin{lemma}\label{prop-capa}
In any set $\Lambda\subset \Z^d$, there is a subset $U$, whose
capacity is maximal (of order its volume), and with volume of order 
$|\Lambda|^{1-2/d}$.
\end{lemma}
Let us emphasize that the notion of capacity plays a key role here,
even though it does not appear in the statement of Theorem \ref{theo-Kn}.

From any set $\Lambda$ we extract a much smaller subset 
as hard to cover, and this subset has maximal capacity.
Lemma \ref{prop-capa} is useful, when we combine it with an older
estimate from \cite{AS} (see Proposition \ref{prop.cap.old} below) 
which states that the probability of filling a region on
a given scale, and at a given density $\rho$, is bounded
by exponential of minus its capacity times $\rho$.

The second idea is as elementary and deals with 
deviations for the sum of an adapted process. An instance of the result that we use is as follows (see Proposition \ref{prop-LD-transfer.general} for a more general statement):
\begin{proposition}\label{prop-LD-transfer}
There exists a positive constant $c$ such that
for any sequence $\{X_i\}_{i\in \N}$ of nonnegative random variables
adapted to a filtration $\{\F_i\}_{i\in \N}$, and almost surely bounded
by $1$, it holds for any $n\ge 1$ and any $\zeta>0$, 
\begin{equation}\label{main-transfer}
\bP\left(\sum_{i=1}^n X_i>\zeta\right)\le
\bP\left(\sum_{i=1}^n \bE[X_i \mid \F_{i-1}] >c 
\zeta\right)+\exp(-c\zeta).
\end{equation}
\end{proposition}
Thus, without any condition on the variance of the $\{X_i\}_{i\in \N}$,
the upward deviation of their sum is comparable to that
of a conditioned sum, up to a multiplicative error. 

\paragraph{Plan of the paper.}
Section~\ref{sec-notation} introduces notation, and recall
useful known results.
Section~\ref{sec-tools} highlights Lemma~\ref{lem-cap}, which
permits to reach the nearly optimal
Proposition~\ref{prop-cap}, and explains
how it is used to establish Theorem~\ref{theo-Kn}.
Section~\ref{sec-trans} transfers the deviation estimate
from the range to a double sum over Green's function, and
Corollary~\ref{cor-trans} is valid in any dimensions larger than two. 
Section~\ref{sec-ub} is the technical heart of
the paper, and after introducing the scale $T$, in $d\ge 5$ and
in $d=3$, treats the deviation of the
double sum over Green's function. In this section, we establish
the upper bounds in \reff{bounds-d5}
and \reff{theo.d3}. The lower bounds are proved in
Section~\ref{sec-LB}. Finally
Section~\ref{sec-gauss} deals with the Gaussian regime and the proof of Theorem \ref{theo-gauss}.


\section{Notation and basic results}\label{sec-notation}
In this section, we introduce further notation,
and recall Lemma~\ref{lem.sum.Green} from \cite{AS}.

We discuss only $d\ge 3$ in this study.
For $z\in \Z^d$, We denote by $\bP_z$ the law of the 
simple random walk starting at $z$, and simply by $\bP$ when $z=0$.
Green's functions read as follows. For $z\in \Z^d$, 
\begin{equation}\label{def-Green}
G(z):=\mathbb E\left[\sum_{n=0}^\infty {\bf 1}\{S_n=z\}\right],
\quad\text{and}\quad\forall T\in \N,\quad
G_T(z):=\mathbb E\left[\sum_{n=0}^T {\bf 1}\{S_n=z\}\right].
\end{equation}
In particular, if $\cR_\infty$ is the range over period $[0,\infty)$,
\begin{equation}\label{inGT}
\mathbb P(z\in \cR_\infty) \le G(z)\quad\text{and}\quad
\mathbb P(z\in \cR_T) \le G_T(z).
\end{equation}
The following asymptotic is well known (see \cite{LL}):  
\begin{equation}\label{Green}
G(z)\ =  \ \cO\left( \frac 1{\|z\|^{d-2}}\right),  
\end{equation}
with $\|\cdot \|$ the Euclidean norm. 
Another result we should need concerns the sum of $G_T$ 
in a region of space with prescribed density. 
For this purpuse, we introduce the local times. For
$n\in \N\cup\{\infty\}$, and $\Lambda \subseteq \Z^d$, we write 
$$
\ell_n(\Lambda):= \sum_{k=0}^n \1\{S_k\in \Lambda \}.
$$
We quote a result which involves no randomness,
but we state it in this form which we use later.
\begin{lemma}[\cite{AS}, Lemma 2.2] \label{lem.sum.Green}
Assume that for some $\K\subseteq \{1,\dots,n\}$, and $r\ge 1$, 
$$\ell_n(Q(S_k,r)) \le \rho r^d,\quad \text{for all }k \in \K.$$ 
Then for some constant $C>0$ (independent of $r$ and $\K$), 
for any $z\in \Z^d$, 
\begin{equation}\label{time-spent}
\sum_{k\in \K\, :\, S_k\notin Q(z,r)} G_T(S_k-z) \le C \rho \, T.
\end{equation}  
\end{lemma}
The following connection between capacity and Green's function 
is useful. For any finite $\Lambda \subset \Z^d$,  
\begin{equation}\label{cap.Lambda.time}
\cc{\Lambda}\cdot \sup_{y\in \Z^d} \bE_y[\ell_\infty(\Lambda)]\ge |\Lambda|.
\end{equation}
This is a simple consequence of a last exit decomposition
(see Proposition 4.6.4 in \cite{LL}). 
First, for $A\subseteq \Z^d$, we denote by $|A|$ the cardinality of $A$, 
and by 
$$
H_A:=\inf\{n\ge 0\, :\, S_n\in A\},
\quad \text{and}\quad H_A^+:= \inf\{n\ge 1\, :\, S_n\in A\},
$$
respectively the hitting time of $A$ and the first return time to $A$. 
Then, for any $x\in \Lambda$, 
$$
1=\bP_x(H_\Lambda<\infty) = \sum_{y\in \Lambda}G(y-x)\bP_y(H_{\Lambda}^+
=\infty).
$$
Summing over $x\in \Lambda$, gives 
$$
|\Lambda| = \sum_{y\in \Lambda} \bP_y(H_{\Lambda}^+=\infty) \sum_{x\in \Lambda}G(y-x) = \sum_{y\in \Lambda} \bP_y(H_{\Lambda}^+=\infty) \cdot \bE_y[\ell_\infty(\Lambda)], 
$$
and \reff{cap.Lambda.time} follows using \eqref{cap.def}.

We denote the range between two times $m\le n$, 
as $\RR[m,n]:=\{S_m,\dots,S_n\}$. 
We recall that by the inclusion-exclusion formula, one has for any integers $n$ and $m$, 
\begin{equation}\label{king-2}
|\cR_{n+m}|=|\cR_n|+|\cR[n,n+m]| -|\cR_n\cap\cR[n,n+m]|.
\end{equation}
If we write $f\asymp g$, when $f/g$ is bounded from above 
and below by positive constants, one has 
\begin{equation}\label{esperance.intersection}
\bE[|\RR_n\cap \RR[n,2n]| ] \asymp 
\bE[|\RR_n\cap \RR[n,\infty)| ] \asymp \left\{ 
\begin{array}{ll}
\sqrt{n} & \text{if }d=3\\
\log n & \text{if }d=4\\
1& \text{if }d\ge 5. 
\end{array}
\right. 
\end{equation} 
We now recall asymptotics on the variance of the 
range (see \cite{JO,JP}).
\begin{lemma} \label{lem.variance}
In dimension three $\var(|\cR_n|)= \cO(n\log n)$,
and for $d\ge 4$, $\var(|\cR_n|)= \cO(n)$.
\end{lemma}


\section{Main Tools}\label{sec-tools}
We present here the proof of Theorem \ref{theo-Kn}. 
As a byproduct we obtain that on the event 
where the size of the set $\K_n(\cdot,\cdot)$ (see \reff{def-Kn}) is large, 
the time spent in a related set $\mathcal V_n(\cdot,\cdot)$, 
as in \reff{def-CV}, is also large.
Namely, we get the following result 
which enters the proof of Theorem \ref{theo-path}. 
\begin{proposition}\label{prop.Kn.path}
For any $A>0$, there exist $\alpha>0$, such that for any $n\ge 1$, $L>0$, $r\ge 1$, $\rho>0$, 
$$\bP\left( \K_n(r,\rho)\ge L,\  \ell_n(\mathcal V_n(r,2^{-d}\rho))\le  \alpha L\right) \le \exp(-A\rho^{2/d}L^{1-2/d}). $$
 \end{proposition}  
Before we come to the proofs, let us first introduce useful tools and notation. For $r>0$, we define 
$$
\mathcal A(r):=  \{\mathcal C\subset \Z^d\, :\, 
|\mathcal C|<\infty,\, \|x-y\|_\infty>2r
\text{ for all }x\neq y\in \mathcal C\},
$$ 
where $\|\cdot\|_\infty$ is the sup-norm.
One reason for the choice of a separation of $2r$ is that if 
$x,y,z\in\Z^d$ with overlapping cubes $Q(x,r)\cap Q(y,r)\not=\emptyset$ and
$Q(x,r)\cap Q(z,r)\not=\emptyset$, then $ \|z-y\|_\infty\le 2r$.

The first step towards the proof of Theorem \ref{theo-Kn} 
is the following result.
\begin{lemma}\label{lem-cap}  
There exists a constant $\kappa\in (0,1)$, such that
for any $r\ge 1$ and any $\mathcal C\in \mathcal A(r)$, 
there exists $U\subseteq \mathcal C$, 
satisfying 
\begin{equation}\label{cap.U}
 \kappa |\CC|^{1-\frac 2d} \le |U| \le |\CC|^{1-\frac 2d} ,\quad\text{ and }\quad
\cc{\cup_{x\in U} Q(x,r)}\ge \kappa r^{d-2}|U|.
\end{equation}
\end{lemma} 
\begin{proof}
The lower bound on the capacity 
follows from recalling \reff{cap.Lambda.time}, and 
observing the following fact.
There exists $C>0$, such that for any finite $\CC\in \mathcal A(r)$, there exists $U\subseteq \CC$, satisfying 
\begin{equation}\label{wish-cap}
\frac 1{2^d} |\CC|^{1-\frac 2d}\le |U| 
\le  |\CC|^{1-\frac 2d},\quad \text{and}\quad 
\sup_{y\in\Z^d} \bE_y[\ell_\infty(\cup_{x\in U}Q(x,r))]\le Cr^2.
\end{equation}
Indeed, define first $R:=\lfloor r|\CC|^{2/d^2}\rfloor$, 
and then recursively a sequence $(z_i)_{i=1,\dots, N}$, 
of elements of $\CC$, such that the cubes 
$Q(z_i,R/2)$, $i=1,\dots,N$, are disjoint, 
and the union of the cubes $Q(z_i,R)$ contains $\CC$. 
Note that in this case $\cup_{x\in \CC}Q(x,r)$ 
is contained in $\cup_{i\le N} Q(z_i,2R)$, 
which implies $N(2R)^d\ge |\CC|r^d$, 
since by definition the cubes $Q(x,r)$, for $x\in \CC$, are disjoint. 
By definition of $R$, this gives $N\ge 2^{-d}|\CC|^{1-2/d}$. We now set $N_0:=\min (N,\lfloor |\CC|^{1-\frac 2d}\rfloor)$, and let $U:=\{z_1,\dots,z_{N_0}\}$. 
By construction the cardinality of $U$ satisfies the desired constraints from 
\reff{wish-cap}. Furthermore, using \eqref{Green} we verify 
that for any $y\in \Z^d$,  
$$ 
\bE_y[\ell_\infty(\cup_{x\in U}Q(x,r))] 
\le \sum_{x\in U}\sum_{z\in Q(x,r)}G(z-y)
\le Cr^2 + C\frac{r^d}{R^{d-2}}|U|^{2/d}\le 2Cr^2,
$$
as expected in \reff{wish-cap}, and this concludes the proof of the lemma. 
\end{proof}

An important consequence of the previous lemma is the following result.

\begin{proposition}\label{prop-cap}
There exist positive constants $C_0$ and $\kappa$, such that 
for any $r\ge 1$, $\rho>0$, $m\ge 1$, $n\ge 1$, satisfying the
condition $\rho r^{d-2}\ge C_0\log n$, one has 
$$
\bP\left(\exists \C \in \A(r)\, :\, |\C|=m, 
\text{ and }\ell_n(Q(x,r))\ge \rho r^d \ 
\text{for all x }\in \C\right) 
\le C \exp\left(-\kappa \rho (mr^d)^{1-\frac 2d}\right).
$$
\end{proposition}
The proof of this proposition is based on Lemma \ref{lem-cap} and an analogous result from \cite{AS}: 

\begin{proposition}[\cite{AS}] \label{prop.cap.old}
There exist positive constants $C$ and $\kappa$, such that 
for any $\rho>0$, $r\ge 1$, $\C\in \A(r)$, and $n\ge 1$, one has 
\begin{equation}\label{main-AS17}
\bP\left(\ell_n(Q(x,r))\ge \rho r^d \ \text{for all x }\in \C\right) 
\le C (n|\C|)^{|\C|} \exp\left(-\kappa\rho 
\ \cp(\bigcup_{x\in \C} Q(x,r)) \right).
\end{equation}
\end{proposition}

\begin{remark}\emph{Recall that  one has the general lower bound $\cc{\Lambda}\ge c| \Lambda|^{1-2/d}$, with $c>0$ some constant. Thus in the previous proposition, 
the term in the exponential is also at least of the order of 
$(|\C|r^d)^{1-2/d}$, 
as in Proposition \ref{prop-cap}, and it is the largest bound
one obtains without additional knowledge on the set $\C$. 
Therefore the main input of our new Proposition \ref{prop-cap} 
when compared 
to the Proposition \ref{prop.cap.old} is that the combinatorial term in front of the exponential has been completely removed. This is of 
crucial importance for the proof of Theorem \ref{theo-Kn}.}
\end{remark} 

\begin{proof}[Proof of Proposition \ref{prop-cap}]
Applying Lemma~\ref{lem-cap}, we see that from any set $\C\in \A(r)$ with cardinality $m$, 
one can extract a subset $U\subseteq \C$, satisfying \eqref{cap.U}. In particular its cardinality is of order $m^{1-2/d}$, and thus 
the number of possible choices for such set $U$ in $[-n,n]^d$ 
is at most $\exp(d m^{1-2/d} \log 2n)$.
It suffices then to use the lower bound on the capacity in \eqref{cap.U} and apply Proposition \ref{prop.cap.old} for such $U$, 
since by the hypothesis \eqref{hyp.r} all the combinatorial factors 
can be absorbed in the exponential bound given by \eqref{main-AS17}.
\end{proof}

We arrive now to the proof of Theorem \ref{theo-Kn}. 
As an intermediate step, we prove the same result for the sets $\K_n^*(r,\rho)$ defined by: 
$$\K_n^*(r,\rho):= \{k\le n:\ \rho r^d \le \ell_n(Q(S_k,r))\le 2 \rho r^d\}.$$
\begin{lemma}\label{lem-Kn}
There exist positive constants $C_0$ and $\kappa$, 
such that for any $\rho>0$, $r\ge 1$, and $n\ge 2$, satisfying  
\begin{equation*}
\rho\, r^{d-2} \ge C_0\log n,
\end{equation*} 
one has for any $L\ge 1$, 
\begin{equation*}
\mathbb P\big(|\K_n^*(r,\rho)|\ge L\big)\, 
\le \, C \exp\left(-\kappa\, \rho^{\frac 2d}\, L^{1-\frac 2d} \right).
\end{equation*}
\end{lemma}
\begin{proof}[Proof of Theorem \ref{theo-Kn}]
It suffices to observe that for any $\gamma<1$
\begin{eqnarray*}
\{|\K_n(r,\rho)|\ge L\} \, \subseteq \, \bigcup_{i\ge 0} 
\left\{|\K_n^*(r,2^i\rho)|\ge c\frac{L}
{2^{\frac{2\gamma}{d-2}\cdot i}}\right\},
\quad\text{with}\quad c:=1/(1-2^{-\frac{2\gamma}{d-2}}). 
\end{eqnarray*}
The result follows by a union bound and Lemma~\ref{lem-Kn}. 
\end{proof}
\begin{remark}\label{rem-highLS} \emph{Note that when $\gamma=1$, all subsets
$\{|\K_n^*(r,2^i\rho)|\ge cL 2^{-\frac{2}{d-2}\cdot i}\}$
have the same upper bound through Lemma~\ref{lem-Kn}.
However, with $\gamma<1$, and in view of the lower bounds
of \reff{bounds-d5} and \reff{theo.d3}, the contributions
with $i$ larger than some fixed integer are negligible.}
\end{remark}

It remains to prove Lemma~\ref{lem-Kn}. 
\begin{proof}[Proof of Lemma \ref{lem-Kn}] 
We claim that on the event $\{|\K_n^*(r,\rho)|\ge L\}$, 
there is a subset $\C\in \A(r)$, such that 
\begin{equation}\label{cond-C}
(i)\quad|\C|\ge \frac{1}{2\cdot 4^d}\cdot\frac{L}{\rho r^d},\quad
\text{and}\quad
(ii)\quad \ell_n(Q(x,r))\ge  \rho r^d, \quad\forall x\in \C.
\end{equation}
To see this, consider 
$k_1:=  \inf \K_n^*(r,\rho)$, and define inductively for $i\ge 1$, 
$$k_{i+1} := \inf \K_n^*(r,\rho)\bs \{ k :  S_k\in  
\bigcup_{j\le i}Q(S_{k_j},2r)\}, $$
with the convention that $k_{i+1}=\infty$, when the set above is empty. 
Note that by definition of $\K_n^*(r,\rho)$, for any $x\in \Z^d$, one has 
$$
|Q(x,r/2)\cap \{S_k\, :\, k\in  \K_n^*(r,\rho)\}| \le 2\rho r^d.
$$ 
Thus for each $i$, such that $k_i\le n$, one can tile up 
the cube $Q(S_{k_i},2r)$ with at most $4^d$ cubes of length $r/2$, 
and deduce that  
$$
|Q(S_{k_i},2r)\cap\{S_k\, :\,k\in\K_n^*(r,\rho)\}|\le 4^d\cdot 2 \rho r^d.
$$
Now, we define $\C$ as the set containing all the $S_{k_i}$, with $k_i$ finite.
By construction, on the event $\{|\K_n^*(r,\rho)|\ge L\}$, 
\reff{cond-C} holds.

Note now that if $\rho r^{d-2}\ge C_0\log n$, 
then the hypothesis of Proposition \ref{prop-cap} is satisfied,
and we get 
\begin{align*}
\bP\left(|\K_n^*(r,\rho)|\ge L\right) &\le  \bP\left(\exists \C\in \A(r)\, :\,  |\C|\ge
\frac{1}{2\cdot 4^d}\cdot\frac{L}{\rho r^d},\ \text{ and }\ \ell_n(Q(x,r)\ge \rho r^d\ \ \text{for all } x\in \C \right)    \\
& \le  C   \exp\left(- \kappa \rho 
\big(\frac{L}{2\cdot 4^d\rho }\big)^{1-2/d}\right) 
\le C\exp(-\kappa' \rho^{2/d} L^{1-2/d}),
\end{align*}
with $\kappa'>0$ some other constant, proving the Lemma. 
\end{proof}

Finally we give a proof of Proposition \ref{prop.Kn.path}.

\begin{proof}[Proof of Proposition \ref{prop.Kn.path}]
Define for $n\ge 1$, $\rho>0$, $L>0$, $r\ge 1$, and $\alpha>0$, 
$$\A_n(r,\rho,L,\alpha):= \left\{\exists \C \in \A(r) \, :\, |\C|\ge \frac{\alpha L}{\rho r^d} \text{ and } \ell_n(Q(x,r))\ge \rho r^d \ \forall x\in \C\right\}. $$ 
Fix some $A>0$. The proof of Lemma \ref{lem-Kn} and Remark~\ref{rem-highLS} give some $\alpha>0$, such that
$$\bP\left( \left\{\K_n(r,\rho)\ge L\right\} \cap \A_n^c(r,\rho,L,\alpha)\right) \le \exp(-A\rho^{2/d}L^{1-2/d}). $$
Suppose now that the event $\A_n(r,\rho,L,\alpha)$ holds. Consider the random
$\C$ which realizes this event. 
For $x\in \C$, there is $z\in r\Z^d$,
such that $Q(x,r)\cap Q(z,r)\not=\emptyset$, and $\ell_n(Q(z,r))\ge 
2^{-d}\rho r^d$. This is because the cube $Q(x,r)$ is inside $2^d$
neighboring cubes with centers in $r\Z^d$. Note that by definition
of $\A_n(r,\rho,L,\alpha)$, this site $z$ cannot be chosen by another site of $\C$.
Thus, using notation $\C_n$ as defined in \reff{def-CV}, 
\begin{equation*}
\A_n(r,\rho,L,\alpha)\subseteq\{|\C_n(r,2^{-d}\rho)|\ge
\frac{\alpha L}{\rho r^d}\}
\subseteq \{\ell_n(\mathcal V_n(r,2^{-d}\rho))\ge  \alpha L\}.
\end{equation*}
This concludes the proof.
\end{proof}


\section{Transferring Deviations to the Corrector}\label{sec-trans}
We present here a general concentration result that 
allows to express downward deviations for the centered 
volume of the range, as upward deviations  
for a so-called corrector term which reads, for a time-scale $T$,
\begin{equation}\label{def-corrector}
\sum_{k=0}^n \sum_{x\in \cR_k} \frac{1}{T} G_T(x-S_k).
\end{equation}
Our main result replaces the Doob-type martingale decomposition of
\cite{AS}, and has the important novelty
that it no more requires some delicate bounds on the 
conditional variance of the intersection of two ranges. 
This is crucial in choosing the time scale $T$, 
which impacts the size of the moderate deviations we cover. 
Our result is a consequence of the following fact, generalizing Proposition \ref{prop-LD-transfer}. 

\begin{proposition}\label{prop-LD-transfer.general}
There exists a positive constant $c$ such that
for any sequence $\{X_i\}_{i\in \N}$ of nonnegative random variables
adapted to a filtration $\{\F_i\}_{i\in \N}$, and almost surely bounded
by $1$, it holds for any $T\ge 1$, $n\ge T$, and $\zeta>0$, 
\begin{equation*}
\bP\left(\frac 1 T \sum_{i=T}^{n} X_i>\zeta\right)\le
\bP\left(\frac 1T \sum_{i=T}^{n} \bE[X_i \mid \F_{i-T}] >c 
\zeta\right)+\exp(-c\zeta).
\end{equation*}
\end{proposition}
When applied to the size of the range, and using a classical decomposition that we shall recall later, we obtain the following corollary (with $\sigma_T:=T$ in dimension $d\ge 4$, and $\sigma_T := T\log T$ in dimension $d=3$). 

\begin{corollary}\label{cor-trans}
There exists a positive constant $c$, such that for any $n\ge 1$, any $2\le T\le n$, and any $\zeta>n/\sigma_T$, with 
\begin{equation}\label{dev.corollaire}
\bP\left( |\RR_n| - \bE[|\RR_n|] \le - \zeta \right) \le  2T\exp(-c\frac{\zeta}{T}) + \bP\left( \sum_{k=0}^n \sum_{x\in \cR_k} \frac{1}{T} G_T(x-S_k)\ge c \zeta\right).
\end{equation}
\end{corollary}

\begin{remark}\label{rem.5/7}
\emph{In dimension $3$, the mean of the sum $\sum_{k=0}^n \sum_{x\in \cR_k} \frac{1}{T} G_T(x-S_k)$, is of order $n/\sqrt{T}$. Thus the second term in the right-hand side 
of \eqref{dev.corollaire} starts to decay only for $\zeta>n/\sqrt T$. On the other hand, for the first term to be of the right order, one needs to take $T$ at most of order $(\zeta n)^{1/3}$. This imposes the condition 
$\zeta > n^{5/7}$ (up to constant), and explains why our techniques cannot cover the whole non Gaussian regime in dimension $3$. }
\end{remark}

Proposition~\ref{prop-LD-transfer.general} is based on three simple facts 
that we have gathered in
three lemmas. 

\begin{lemma}\label{lem-trans1}
Let $Z$ be some nonnegative random variable bounded by 1, and $\G$ be a
$\sigma$-field. There exists a positive constant $\kappa^*$, independent of $Z$ and
$\G$, such that almost surely
\begin{equation*}
\bE\left[\exp\big( Z -\kappa^* \bE[Z\mid \G]\big)\mid \G \right]\le 1.
\end{equation*}
\end{lemma}
\begin{proof}
Set $g(t)=(e^t - 1 -t)/t^2$, and observe that
\begin{equation*}
\bE[\exp(Z)\mid \G] = 1+ \bE[Z\mid \G]+ \bE[Z^2 g(Z)\mid \G].
\end{equation*}
Note then that on $\R^+$, the function $g$ is increasing, and positive so that
$0\le g(Z)\le g(1)$. Thus, defining $\kappa^*:=1+g(1)$, we have
\begin{equation*}
\bE[\exp(Z)\mid \G] \le 1+\kappa^* \bE[Z\mid \G]\le \exp(\kappa^* \bE[Z\mid \G]).
\end{equation*}
This concludes the proof.
\end{proof}
We now extend this result to a sum of random variables adapted to some filtration. 
\begin{lemma}\label{lem-trans2}
Let $\{Z_i\}_{i\in \N}$ be adapted to some filtration
$\{\G_i\}_{i\in \N}$, and such that for each $i\in \N$, almost surely $0\le Z_i\le 1$.
Then for any $n\ge 1$, 
\begin{equation}\label{ineq-trans2}
\bE\left[\exp\left( \sum_{i=1}^n Z_i -\kappa^* \sum_{i=1}^n
\bE[Z_i \mid \G_{i-1}]\right)\right]\le 1.
\end{equation}
\end{lemma}
\begin{proof}
The proof is immediate by induction. Indeed, the result for $n=1$ is exactly Lemma \ref{lem-trans1}. Assuming now that \eqref{ineq-trans2} holds for some $n$, we get the result at time $n+1$, by first applying Lemma \ref{lem-trans1} to $Z_{n+1}$ and $\G_n$, and then by using the induction hypothesis. 
\end{proof}
Proposition~\ref{prop-LD-transfer.general} requires one more step,
which is the content of the next lemma.
\begin{lemma}\label{lem-trans3}
Let $\{\cS_j\}_{j\in \N}$ and $\{\tilde \cS_j\}_{j\in \N}$ be two sequences of random variables, such that
\[
\bE[\exp\big( \cS_j-\tilde \cS_j\big)]\le 1, \qquad \forall j\in \N.
\]
Then, for any integer $T$, and any $\zeta>0$, 
\begin{equation*}
\bP\left( \frac{1}{T} \sum_{j=1}^T \cS_j> \zeta\right)\le
e^{-\zeta/2}+
\bP\left( \frac{1}{T} \sum_{j=1}^T \tilde \cS_j> \frac{\zeta}{2}\right).
\end{equation*}
\end{lemma}
\begin{proof}
It suffices to observe that by Jensen's inequality,
\begin{equation*}
\bE\left[\exp
\Big(\frac{1}{T} \sum_{j=1}^{T}(\cS_j-\tilde \cS_j)\Big) \right]\le
\frac{1}{T}\sum_{j=1}^{T} \bE\left[\exp
\big(\cS_j-\tilde \cS_j\big)\right]\le 1.
\end{equation*}
Then the result follows by a union bound and Chebyshev's exponential inequality. 
\end{proof}

We are now ready to prove Proposition~\ref{prop-LD-transfer.general}.

\begin{proof}[Proof of Proposition~\ref{prop-LD-transfer.general}]
Let us assume to simplify notation that $n$ is of the form $n=mT-1$, for some integer $m\ge 2$, and leave the necessary (minor) changes to the proof to the reader for a general $n$. 
For $j=0,\dots,T-1$, define 
$$\cS_j:=  \sum_{i=1}^{m-1} X_{j+iT}, \quad \text{and}\quad \tilde \cS_j:=  \kappa^*\sum_{i=1}^{m-1} \bE[X_{j+iT}\mid \cF_{j+(i-1)T}],$$
with the same constant $\kappa^*$ as in Lemma \ref{lem-trans3}. Note first that by applying Lemma \ref{lem-trans2}, for any fixed $0\le j\le T-1$, 
with $Z_i=X_{j+iT}$, and $\G_i:=\F_{j+iT}$, we get that 
$$\bE[\exp(\cS_j-\tilde \cS_j)] \le 1,$$
for any $0\le j\le T-1$. Then the proposition follows from Lemma \ref{lem-trans3}, taking $c=1/(2\kappa^*)$. 
\end{proof}

\begin{proof}[Proof of Corollary~\ref{cor-trans}]
Fix some $2\le T\le n$, and use \eqref{king-2} repeatedly, to obtain for any $0\le j\le T-1$, 
\begin{equation}\label{dec.fluct}
|\RR_n| = U_j  - \sum_{i=2}^{\lfloor n/T\rfloor } |\RR[j+(i-1)T+1,j+iT]\cap \RR[j,j+(i-1)T]|+\varepsilon(T),
\end{equation}
where $ |\epsilon(T)|\le T$, and 
$$U_j:=\sum_{i=1}^{\lfloor n/T\rfloor } |\RR[j+iT-(T-1),j+iT]|,$$
is a sum of independent and identically distributed terms. Note now that if $T>\zeta/2$, the desired result is immediate, so that one can assume now that $T\le \zeta/2$. 
Then we get 
$$\bP\left(|\RR_n| - \bE[|\RR_n|] \le -\zeta\right) \le \bP\left(\frac 1T \sum_{j=0}^{T-1} (U_j - \bE[U_j]) \le -\frac{\zeta}{4}\right) + \bP\left(\frac 1T \sum_{i=T}^{n'} X_i \ge \frac{\zeta}{4T}\right), $$
where $n':=T\lfloor n/T\rfloor -1$, and for $i\ge T$, 
$$X_i:= \frac{|\RR[i-T+1,i] \cap \RR_{i-T}|}{T}.$$ 
Using a union bound, and then Bernstein's inequality together with Lemma \ref{lem.variance}, we get  
$$\bP\left(\frac 1T \sum_{j=0}^{T-1} (U_j - \bE[U_j]) \le -\frac{\zeta}{4}\right) \le \sum_{j=0}^{T-1} \bP\left(U_j - \bE[U_j] \le -\frac{\zeta}{4}\right) \le T \exp( -c\frac{\zeta}{T}). $$
On the other hand, applying Proposition \ref{prop-LD-transfer.general}, we get 
$$\bP\left(\frac 1T \sum_{i=T}^{n'} X_i \ge \frac{\zeta}{4T}\right) \le \exp(-c\frac{\zeta}{T}) + \bP\left( \sum_{i=T}^{n'} \bE[X_i\mid \cF_{i-T}] \ge c\zeta\right). $$
To conclude, it suffices to observe that by \eqref{inGT}, one has for any $i$, 
$$\bE[X_i\mid \cF_{i-T}] \le \sum_{x\in \RR_{i-T}} \frac 1TG_T(x- S_{i-T}). $$
\end{proof}


\section{Upper bounds in \eqref{bounds-d5} and \eqref{theo.d3} 
and proof of Theorem \ref{theo-path}}\label{sec-ub}

In this section,
we prove the upper bounds in \eqref{bounds-d5} 
and in \eqref{theo.d3}, as well as Theorem \ref{theo-path}  
(admitting for a moment the lower bounds in Theorem \ref{theo.d5} and in 
\eqref{theo.d3}, which we prove later using an independent argument). 
The problem here is of estimating deviation of the 
corrector \reff{def-corrector}. The strategy
we follow is naive: we decompose space into level sets of
the local times, and estimate each contribution to Green's function.
However, there is a twist: in our multi-scale analysis, 
space-scale and occupation density are linked through \reff{hyp.r}.

\subsection{Dimensions Five and Larger}
We assume here that $d\ge 5$. 
We fix in this whole subsection the value of $T$ as 
\begin{equation}\label{def.T}
T:= \lceil \frac{c}{2\underline \kappa}\cdot \zeta^{2/d}\rceil,
\end{equation}
where $\underline \kappa$ is as in Theorem \ref{theo.d5}, 
and $c$ as in Corollary \ref{cor-trans}, 
so that the factor $\exp(-c\zeta/T)$  
appearing in the latter is negligible 
compared to the lower bound estimate
for the probability of the event $\{|\RR_n|-\bE[|\RR_n|]\le - \zeta\}$. 
Note that with this definition of $T$, the hypotheses of Corollary~\ref{cor-trans}  
are satisfied for any $\zeta > n^{\frac d{d+2}}\log n$, 
and $n$ large enough.

Then for $i\ge 1$ we consider $\rho_i:=2^{-i+1}$, and
define the {\it associated} length-scale $r_i$ by   
\begin{equation}\label{def.ri}
\rho_i \cdot r_i^{d-2}=C_0 \log n,
\end{equation}
with the same constant $C_0$ as in \eqref{hyp.r}. 
The length $r_i$ is the smallest scale on which we can probe density $\rho_i$.

Recall the definition of the sets $\K_n(r,\rho)$ from the introduction. Then define $\hat \K_1:= \K_n(r_1,\rho_1)$, and for $i\ge 2$,
\begin{equation*}
\hat \K_i=\K_n(r_i,\rho_i)\bs \bigcup_{j<i} \K_n(r_j,\rho_j).
\end{equation*}
Note that, the peculiarity of $\hat \K_i$, for $i>1$, 
is that for any $k\in \hat \K_i$,
the time spent on $Q(S_k,r_{i-1})$ 
is less than $\rho_{i-1}r_{i-1}^d$. 
Now for $A>0$, $\delta>0$, and $I\ge 1$ integer, we define 
\begin{equation*}
\cE(A,\delta,I) : = \left(\bigcap_{1\le i\le I} 
\left\{|\hat \K_i| \le \delta L_i\right\} \right)\cap 
\left(\bigcap_{i>I} \left\{|\hat \K_i| \le AL_i\right\}\right),\quad
\text{with}\quad  
L_i:=\frac{\zeta}{\rho_i^{2/(d-2)}}, \ \forall i\ge 1. 
\end{equation*}
We state now the main result of this subsection, and then explain how it implies both the upper bound in \eqref{bounds-d5} and Theorem \ref{theo-path} (for the case $d\ge 5$). 
\begin{proposition}\label{prop.up.5}
There exists a constant $C_0>0$, such that the following holds. 
For any $A>0$, there exist an integer $I\ge 1$, 
and $\delta>0$, such that for any $C_0 n^{\frac{d}{d+2}}\cdot (\log n)^{\frac{2d}{d^2-4}} \le \zeta \le n$,  
$$
\cE(A,\delta,I) \ \subseteq \ 
\left\{\sum_{k=0}^n \sum_{x\in \cR_k} \frac{1}{T} G_T(x-S_k)
\le \zeta\right\}.
$$
\end{proposition}

\begin{proof}[Proof of the upper bound in \eqref{bounds-d5}] Consider the constants $\delta$ and $I$, which are associated to $A=1$ in Proposition \ref{prop.up.5}. 
Using Theorem \ref{theo-Kn} and a union bound, we get for any $\zeta\ge \log n$, 
$$\bP(\cE(1,\delta,I)^c) \le C\exp(-\kappa \zeta^{1-2/d}),$$
for some positive constants $C$ and $\kappa$ (note that the number of indices $i$, such that $\hat \K_i$ is nonempty is at most of order $\log n$).  
Then the result follows from Corollary \ref{cor-trans} and Proposition \ref{prop.up.5}, using the definition \eqref{def.T} of $T$. 
\end{proof}

\begin{proof}[Proof of Theorem \ref{theo-path}] 
We assume here the lower bound in \eqref{bounds-d5}, and start with the proof of \eqref{result-pathwise}. By (the proof of) Corollary \ref{cor-trans}, and our choice of $T$ in \eqref{def.T}, 
one can see that for any sequence $\{\zeta_n\}_{n\ge 1}$ satisfying $\zeta_n>\frac{n\log n}{T}$, one has for some constant $c>0$, 
$$\lim_{n\to \infty} \bP\left(\sum_{k=0}^n \sum_{x\in \cR_k} \frac{1}{T} G_T(x-S_k)\ge c\zeta_n \mid |\cR_n|- \bE[|\cR_n|]\le - \zeta_n\right) = 1.$$
Next, by Theorem \ref{theo-Kn}, one can choose $A$ large enough, so that 
$$\lim_{n\to \infty} \bP\left(\cup_{i\ge 1} \{|\hat \K_i|\ge AL_i \}\mid  |\cR_n|- \bE[|\cR_n|]\le - \zeta_n\right) = 0.$$
Then consider the parameters $\delta$ and $I$ associated to this constant $A$ in Proposition \ref{prop.up.5}. Applying this proposition, we get 
$$\lim_{n\to \infty} \bP\left(\cup_{1\le i\le I} \{|\hat \K_i|\ge \delta L_i \}\mid  |\cR_n|- \bE[|\cR_n|]\le - \zeta_n\right) = 1.$$
Now observe that for any $i\le I$, one has 
$$\K_n(r_i,\rho_i)\subseteq \K_n(r_I, 2^{\frac{2}{d-2}(i-I)} \rho_I).$$
We deduce that 
$$\lim_{n\to \infty} \bP\left(|\K_n(r_I, 2^{-\frac{2}{d-2}I} \rho_I)|\ge \delta \zeta_n \mid  |\cR_n|- \bE[|\cR_n|]\le - \zeta_n\right) = 1,$$
and then \eqref{result-pathwise} follows from Proposition \ref{prop.Kn.path}.

Finally the characterization of the capacity in \reff{result-capacity}, 
is a simple consequence of a general result \cite{AS20a}, namely
(1.15) of Theorem 1.5, once we know \reff{result-pathwise}.
\end{proof}

\begin{proof}[Proof of Proposition \ref{prop.up.5}]
Define 
$$\hat \K_\infty := \{0,\dots,n\} \setminus \left(\cup_{i\ge 1} \hat \K_i\right).$$
Note that by definition one can write the set of integers smaller than $n$ as the disjoint union:  
$$ \{0,\dots,n\} = \hat \K_\infty \cup (\bigcup_{i\ge 1} \hat \K_i).$$  
Thus one has 
\begin{equation}\label{decomp-xi}
\sum_{k=0}^n \sum_{x\in \cR_k} \frac{1}{T} G_T(x-S_k)
\le \Sigma_1 + \Sigma_2 + \Sigma_3 + 2\Sigma_4 +\Sigma_5,
\end{equation}
with 
\begin{equation}\label{def-Sigma}
\begin{split}
\Sigma_1:=&\sum_{k\in \hat \K_1} \sum_{x\in \cR_k}
\frac{1}{T} G_T(x-S_k),\\
\Sigma_2:=& \sum_{i\ge 2}\sum_{k\in \hat \K_i} \sum_{k'\in \hat \K_1} 
\frac{1}{T} G_T(S_k-S_{k'})\1\{S_k\notin Q(S_{k'}, r_{i-1})\}, \\
\Sigma_3:=&   \sum_{i\ge 2}\sum_{k\in \hat \K_i} \sum_{x\in \cR_k} 
\frac{1}{T}G_T(x-S_k)\1\{x\in Q(S_k, r_{i-1}) \},\\
\Sigma_4:=&\sum_{i\ge 2}\sum_{k\in \hat \K_i}\sum_{j\ge i}
\sum_{k'\in \hat \K_j} \frac{1}{T} G_T(S_k-S_{k'})
\1\{S_k\notin Q(S_{k'}, r_{i-1})\}, \\
\text{ and}\quad \Sigma_5 := & \sum_{k\in \hat \K_\infty} \sum_{x\in \cR_k}
\frac{1}{T} G_T(x-S_k). 
\end{split}
\end{equation} 
First, by definition of $G_T$, 
one has $\sum_{z\in \Z^d} \frac{1}{T} G_T(z)=1$. 
Thus, on $\cE(A,\delta,I)$, with $\delta\le 1/5$,
\[
\Sigma_1\le |\hat \K_1|\le  \zeta/5.
\]
By definition when $k\in \hat \K_i$ and $i>1$, 
the time spent in $Q(S_k,r_{i-1})$ is smaller than 
$\rho_{i-1} r_{i-1}^d$, so that by Lemma \ref{lem.sum.Green}, 
for some constant $C>0$,  
\begin{equation*}
\Sigma_2 \le 
|\hat \K_1|\sup_z \sum_{i\ge 2}\sum_{k\in \hat \K_i} \frac{1}{T} G_T(S_k-z)\1\{S_k\notin Q(z, r_{i-1})\} 
\le C |\hat \K_1|\sum_{i\ge 2} \rho_{i-1}
\le 2C  |\hat \K_1|\le \zeta/5, 
\end{equation*}
choosing $\delta$ small enough for the last inequality. 
We consider now $k\in \hat \K_i$ for $i\ge 2$, and note that
by definition of $\hat \K_i$,  we have a bound
on the time spent in concentric shells centered around $S_k$, inside $Q(S_k,r_{i-1})$.
Thus, bounding $\sum_i |\hat \K_i|$ by $n+1$, we obtain
\begin{equation*}
\begin{split}
\Sigma_3 & \stackrel{\eqref{Green}}{\le}  C\sum_{i\ge 2}|\hat \K_i| 
\sum_{j=1}^{i-1} \frac{\rho_j r_j^d}{Tr_j^{d-2}} 
\stackrel{\eqref{def.ri}}{\le} C\sum_{i\ge 2}|\hat \K_i| \sum_{j=1}^{i-1}
 \frac{\log n}{Tr_j^{d-4}}\le C \frac{ n \log n}{Tr_1^{d-4}} \stackrel{\eqref{def.ri}}{\le} C\frac{n (\log n)^{\frac{2}{d-2}}}{T}.
\end{split}
\end{equation*}
Using next the hypothesis on $\zeta$ with the constant $C_0$ large enough, we get 
\begin{equation*}
\Sigma_3 \le  \zeta/5. 
\end{equation*}
The same argument gives as well for $n$ large enough, 
$$\Sigma_5 \le \frac{Cn}{T} \left\{r_1^2+ \sum_{j\ge 1} \frac{\rho_j r_j^d}{r_j^{d-2}} \right\} \le C\frac{n(\log n)^{\frac{2}{d-2}}}{T} \le \zeta/5,$$
bounding simply $|\hat \K_\infty|$ by $n+1$, and using for the first inequality that the sum of the Green's function over all points in $Q(0,r_1)$ is of order $r_1^2$, by \eqref{Green}.   
Finally, for any $2\le i \le j$, when  $k'\in \hat \K_j$, 
the time spent on cubes at scale $r_{i-1}$ around such $S_{k'}$, is smaller than $\rho_{i-1}$.
Thus for a constant $C>0$, using again Lemma \ref{lem.sum.Green}, 
\begin{equation*}
\Sigma_4 \le C \sum_{i\ge 2} |\hat \K_i| \rho_{i-1}, 
\end{equation*}
so that on the event $\cE(A,\delta, I)$, 
$$\Sigma_4 \le C\zeta \left(\delta \sum_{1\le i\le I} \rho_{i-1}^{\frac{d-4}{d-2}} + A \sum_{i>I} \rho_{i-1}^{\frac{d-4}{d-2}}\right) \le \frac{\zeta}{10},$$ 
by taking first $I$ large enough, and then $\delta$ small enough. Altogether, 
this proves the proposition. 
\end{proof}


\subsection{Dimension Three}
We assume here that $d=3$, 
and for the same reason than in higher dimension, 
we choose $T$ as 
$$T:=\lceil \frac{c}{2\underline \kappa}\cdot (\zeta n)^{1/3}\rceil,$$
with $\underline \kappa$ as in \eqref{theo.d3}, and $c$ as in Corollary \ref{cor-trans}. 
Then, we recall that $\rho_i=2^{-i+1}$, for $i\ge 1$, 
and define $r_i$ similarly as in higher dimension (with $C_0$ 
as in \eqref{hyp.r}), by $r_i\rho_i =C_0 \cdot\log n$.
We define $I$ to be the smallest integer such that $\rho_I\le c_0\zeta/n$, with a constant $c_0$ that will be fixed in the proof of Proposition \ref{prop-d3}. 
Note that $\zeta/n$ is the correct order of the density of the range we are expecting
under the event $\{|\cR_n|-\bE[|\RR_n|]<-\zeta\}$.
Note also that when $\zeta\ge n^{5/7}\cdot \log n$, then $r_I\le C\sqrt T$, for some constant $C>0$. 
Now we keep the same definition for $\hat \K_i$ as in higher dimension, but only for $1\le i<I$, and we set
\[
\hat \K_I:=\{0,\dots,n\}\bs \bigcup_{j<I} \K_n(r_j,\rho_j).
\]
For $A>0$, $\delta>0$, and $1\le J\le I-1$, an integer, we set 
$$
\cE(A,\delta,J):= \left(\bigcap_{I-J \le i\le I-1} 
\{|\hat \K_i|\le \delta L_i\} \right) \cap \left(\bigcap_{i<I-J}  \{|\hat \K_i|\le A L_i\}\right)
,\quad\text{with}\quad  
L_i:= \frac{\zeta^2}{n\rho_i^2}, \ \forall i\ge 1.
$$
Our main result in this subsection reads as follows. 
\begin{proposition}\label{prop-d3} 
For any $A>0$, there exist $\delta>0$, and an integer $1\le J\le I-1$, such that for any $n\ge 2$, and any $n^{5/7} \cdot \log n \le \zeta\le n$, 
\begin{equation*}
\cE(A,\delta,J) \ \subseteq \ \left\{ \sum_{k=0}^n \sum_{x\in \cR_k} \frac{1}{T} G_T(x-S_k) \le \zeta\right\}. 
\end{equation*}
\end{proposition}
We note that exactly as in higher dimension, one can deduce from this proposition 
the upper bound in \eqref{theo.d3}, as well as Theorem \ref{theo-path}, for the case $d=3$.  

\begin{proof}
We proceed as in higher dimension, and write 
$$\sum_{k=0}^n \sum_{x\in \cR_k} \frac{1}{T} G_T(x-S_k)
\le \Sigma_1 + \Sigma_2 + \Sigma_3 + 2\Sigma_4,$$
with the $(\Sigma_k)_{1\le k\le 4}$ as in \reff{def-Sigma} except that we take only a sum over
$i$ running up to $I$.
Now we assume that $\delta<1/4$, and if $\zeta> n/(4A)$, we set $J=I-1$. Thus in all cases, one has $|\hat \K_1| \le  \zeta/4$, on the event $\cE(A,\delta,J)$. It follows that on this event 
\begin{equation}\label{decomp-8}
\Sigma_1 \le
|\hat \K_1| \le  \zeta/4.
\end{equation}
The term $\Sigma_2$ is treated as in higher dimension. Concerning $\Sigma_3$, we note that for any $k\in \hat \K_i$, with $i\ge 2$, one has 
\begin{equation*}
\sum_{x\in \cR_k}
\frac{1}{T} G_T(S_k-x)\1\{x\in Q(S_k, r_{i-1}) \}\le  
C \sum_{\ell=1}^{i-1} \frac{\rho_\ell r_\ell^2}{T} 
\le C\rho_i,
\end{equation*}
using that $r_\ell^2\le CT$, for all $\ell$.   
Therefore, 
\begin{equation*}
\Sigma_3 \le C\sum_{i=1}^I |\hat \K_i| \rho_{i}.
\end{equation*} 
Bounding simply $|\hat \K_I|$ by $n+1$, we get that on the event $\cE(A,\delta,J)$, with the constant $c_0$ from the definition of $\rho_I$, 
\begin{equation*}
\sum_{i=1}^I |\hat \K_i| \rho_i 
\le \zeta\left(A\sum_{i=1}^{I-J} \frac{\rho_I}{\rho_i} 
+ \delta \sum_{I-J< i< I}  \frac{\rho_I}{\rho_i} + c_0\right)  
\le \zeta(A 2^{-J+1} + 2\delta + c_0).
\end{equation*}
Thus by taking $J$ large enough, and then $\delta$ and $c_0$ small enough, 
we obtain $\Sigma_3\le \zeta/8$ .
An argument as in the proof of Proposition \ref{prop.up.5} 
shows that one has as well
$$
\Sigma_4 \le C\sum_{i=1}^I |\hat \K_i| \rho_{i}\le \zeta/8. 
$$ 
This concludes the proof of the proposition. 
\end{proof}


\section{Lower bounds in \eqref{bounds-d5} and \eqref{theo.d3}}\label{sec-LB}
We prove here the lower bounds in \ref{bounds-d5} and \eqref{theo.d3}. 
We start with the case of dimension $5$ and more which is the easiest. 

\subsection{Dimensions Five and larger}
In this case we obtain the following. Note that the result covers a slightly larger range of values for $\zeta$ than in the statement of Theorem \ref{theo.d5} (the logarithmic factor can be removed).   

\begin{proposition}\label{prop.lower5}
Assume $d\ge 5$. There exist positive constants $\varepsilon$, $K$ and $\underline{\kappa}$, 
such that for any $n\ge 2$ and any $K n^{\frac{d}{d+2}}  \le \zeta \le \varepsilon n$, one has 
$$
\bP\left(|\RR_n| - \bE[|\RR_n|]\le -\zeta\right) 
\ \ge\  \exp\left(- \underline{\kappa} \cdot \zeta^{1-\frac 2d} \right).$$
\end{proposition}
\begin{proof}
The argument is the same as in \cite{AS}, but for convenience let us briefly recall it. The idea is to force one piece of the walk to localize in a small cube of volume $\zeta$. 
More precisely, set $m:=\lfloor \frac{3}{\gamma_d}\zeta\rfloor$. Note first that by \eqref{king-2} and \eqref{esperance.intersection}, 
one has  
$$
|\RR_n| -\bE[|\RR_n|] \le |\RR_m| - \bE[|\RR_m|] + Y,
\quad\text{with}\quad Y:=|\RR[m,n]| - \bE[|\RR[m,n]|] + \cO(1).
$$ 
Note also that $Y$ is independent of $\RR_m$ and that
on the event $\cE:=\{\RR_m \subseteq Q(0,\zeta^{1/d})\}$, 
one has $|\RR_m|\le |Q(0,\zeta^{1/d})| = \zeta$. 
Furthermore, we recall that
$$
\left| \bE[|\RR_m|] - \gamma_d m\right| =\cO(\sqrt m).
$$ 
Thus on the event $\cE \cap \{Y\le 0\}$, 
at least for $\zeta$ large enough, 
$$
|\RR_n| -\bE[|\RR_n|] \le -\zeta.
$$
Finally, recall that on one hand, for some constant $\underline \kappa$, 
$$
\bP(\cE) \ge \exp(-\underline \kappa\cdot  \zeta^{1-2/d}).$$
On the other hand, choosing $\varepsilon$ such that $n-m\ge n/2$, and then using the Central Limit Theorem for the volume of the range \cite{JP,JO}, we deduce that at least for $n$ large enough, 
$$\bP(Y\le 0) \ge 1/4.$$
Since $Y$ and $\cE$ are independent, this concludes the proof of the proposition. 
\end{proof}

\subsection{Dimension three}
Our result covers here as well a larger range of 
values for $\zeta$ than in  \eqref{theo.d3} (recall however, that in dimension $3$ a more general and more precise asymptotic is proved in \cite[Theorem 8.5.2]{Chen}).  
\begin{proposition}\label{prop.lower3}
Assume $d=3$. 
There exist positive constants $\varepsilon$, $\underline{\kappa}$, and $K$, such that for any $n\ge 2$, and any $K n^{2/3} \le \zeta \le \epsilon n$, 
one has 
\begin{eqnarray*}
\bP\left(|\RR_n|-\bE[|\RR_n|] \le - \zeta\right)  \ge \exp(- \underline{\kappa} \cdot (\frac{\zeta^2}n)^{1/3}). 
\end{eqnarray*}
\end{proposition}
\begin{proof} 
One can use also the same argument as in \cite{AS}. Since the proof is quite long, let us only give the main steps, and simply refer to \cite{AS} for details.

Set $m:=\lfloor n/2\rfloor$, and note that  \eqref{king-2} and \eqref{esperance.intersection} now give, 
$$|\RR_n| -\bE[|\RR_n|] \le |\RR_m| - \bE[|\RR_m|] + |\RR[m,n]| - \bE[|\RR[m,n]|] -|\RR_m \cap \RR[m,n]| + \cO(\sqrt n).$$ 
We consider next the event $\cE:=\{\RR_m\subseteq Q((n^2/\zeta)^{1/3})\}$, whose probability is of the right order, since for some constant $\underline \kappa$,  
$$\bP(\cE)\ge \exp(-\underline \kappa \cdot (\zeta^2/n)^{1/3}).$$ 
Moreover, by the results of \cite{AS}, one has for some constant $c>0$, 
$$\bP(|\RR_m|\le \frac{\gamma_d}{2} m ) \le \exp(-cn^{1/3}),$$
which is thus negligible when compared to the probability of $\cE$, if $\zeta\le \epsilon n$, with $\epsilon$ small enough. 
Finally we use the Proposition 4.1 in \cite{AS}, which implies that for any fixed $\Lambda$ in $Q((n^2/\zeta)^{1/3})$, with size of order $n$, the probability that $\RR[m,n]$ 
intersects $\Lambda$ in more than $\zeta$ points, is at least $\exp(-\kappa' \cdot (\zeta^2/n)^{1/3})$, provided $\zeta^3/n^2$ is large enough. 
Applying the latter with $\Lambda = \RR_m$, on the event $\cE\cap \{|\RR_m|\ge \frac{\gamma_d}{2} m\}$, and using known estimates for upward deviations (see for instance \cite[Theorem 8.5.2]{Chen}), concludes the proof of the proposition. 
\end{proof}

\begin{remark}
\emph{We note that rough upper bounds (sufficient for the proof above) for upward moderate deviations of the range can 
also be obtained using soft arguments, so we do not really need to rely on any result of \cite{Chen}. More details are given in \cite{AS19b} (which concerns the capacity of the range, instead of the volume, but the argument works the same in both cases).  }
\end{remark}


\section{The Gaussian Regime}\label{sec-gauss}
We deal here with the Gaussian regime and prove Theorem \ref{theo-gauss}. 
We start with recalling some preliminary estimates from the literature 
in the next subsection. Then we prove 
Theorem \ref{theo-gauss} in Subsection \ref{sec-gaussMDP}. 

\subsection{Preliminary results}

We first state an instance 
of G\"artner-Ellis' Theorem (see Theorem 2.3.6 in \cite{DZ}). 
\begin{theorem}[G\"artner-Ellis]\label{theo-GE}
Let $\{X_n\}_{n\ge 0}$ be a sequence of real random variables.
Assume that for a sequence $\{b_n\}_{n\ge 0}$ going to infinity, and
for any $\theta\in \R$, 
\begin{equation*}
\lim_{n\to\infty} \frac{1}{b_n} \log \bE[\exp(\theta b_n\cdot X_n)]=
\frac{\sigma^2}{2}\cdot \theta^2. 
\end{equation*}
Then, for any $\lambda>0$, 
\begin{equation*}
\lim_{n\to\infty} \frac{1}{b_n}
\log \bP(X_n> \lambda)=-\frac{\lambda^2}{2\sigma^2}.
\end{equation*}
\end{theorem}
The second result we shall need concerns large deviations for random variables with stretched exponential tail.  
\begin{theorem}[A. Nagaev \cite{anagaev}]\label{theo-nagaev}
Let $\{Y_n\}_{n\ge 0}$ be a sequence of centered random variables, such that 
$\bE[\exp(\kappa |Y_1|^\alpha)]<\infty$, for some constants $\kappa>0$, and $\alpha\in (0,1]$. Then 
there are positive constants $c$ and $C$, such that for any $n\ge 1$ and any $t>n^{\frac 1{2-\alpha}}$, 
\begin{equation*}
\bP\big(Y_1+\dots+Y_n>t\big)\le C\exp(-c t^\alpha). 
\end{equation*}
\end{theorem}
In fact the result in \cite{anagaev} (which is also quoted in \cite[Equation (2.32)]{snagaev}) concerns centered random variables with variance $1$ and 
tail distribution equal to $\bP(Y_1>t) = \exp(-\kappa t^{\alpha})(1+o(1))$, for $t$ large enough, but one can easily deduce Theorem \ref{theo-nagaev} from it using stochastic domination.

To be able to use this result in our case, we need
to control the moments of the cross-term, and for that we use a recent result from \cite{AS20b}, which strengthens the previous bound of \cite{KMSS}.

\begin{theorem}[\cite{AS20b}] \label{prop.tworanges}
 Assume that $d\ge 5$. Let $\cR_\infty$ and $\widetilde \cR_\infty$ be the ranges of two
independent simple random walks on $\Z^d$.
There exists a constant $\kappa>0$, such that
\begin{equation*}
\bE\left[\exp\left(\kappa \cdot
|\cR_\infty\cap \widetilde \cR_\infty|^{1-\frac 2d} \right)\right]<\infty.
\end{equation*}
\end{theorem}

\subsection{Proof of Theorem~\ref{theo-gauss}}\label{sec-gaussMDP}
We start with recalling a standard dyadic decomposition
for the volume of the range, which follows from using \eqref{king-2}
repeatedly along a dyadic scheme. For any $L\ge 1$, and $n\ge 2^L$,
\begin{equation}\label{UBd5-1}
|\cR_n|-\bE[|\cR_n|]=  \sum_{i=1}^{2^L} \left( |\cR_i^L| - \bE[ |\cR_i^L|] \right)
-\sum_{\ell=1}^{L}\sum_{i=1}^{2^{\ell-1}}
\left(|\cR_{2i-1}^\ell\cap \cR_{2i}^\ell | -\bE[|\cR_{2i-1}^\ell\cap \cR_{2i}^\ell |]\right),
\end{equation}
where for any fixed $1\le \ell\le L$,
the $\{\cR_i^\ell\}_{i=1,\dots,2^\ell}$, are independent
ranges of length $n2^{-\ell}$ (the time-length is not
exactly equal for each of them since we do not suppose that $n$ is of the form $n=2^K$, for some $K\ge 1$, but they
differ by at most one unit).

Let now $\{\zeta_n\}_{n\ge 0}$ be a sequence as in the statement of Theorem \ref{theo-gauss}, and let $L$ be the integer such that 
$2^{L-1} \le \zeta_n  <2^L$.

We first show that the intersection terms appearing in 
\eqref{UBd5-1} are negligible. Set $Y_i^\ell = |\cR_{2i-1}^\ell\cap \cR_{2i}^\ell | -\bE[|\cR_{2i-1}^\ell\cap \cR_{2i}^\ell |]$. 
Applying Theorems \ref{theo-nagaev} and \ref{prop.tworanges}, we get that for any $\delta>0$, and any $\ell \le L$, 
$$\limsup_{n\to \infty} \frac{n}{\zeta_n^2}\cdot \log \bP\left(\pm \sum_{i=1}^{2^\ell} Y_i^\ell \ge \frac{\delta\zeta_n}{L} \right)  = - \infty. $$ 
By using a union bound we also deduce 
$$\limsup_{n\to \infty} \frac{n}{\zeta_n^2}\cdot \log \bP\left(\pm \sum_{\ell = 1}^L \sum_{i=1}^{2^\ell} Y_i^\ell \ge \delta\zeta_n \right)  = - \infty. $$ 
Thus indeed the intersection terms in \eqref{UBd5-1} can be ignored, and we focus now on proving the Moderate Deviation Principle for the first sum.

For simplicity, let $Z_i:= |\cR_i^L| -\bE[ |\cR_i^L|]$. We apply Theorem \ref{theo-GE} with $X_n:= \frac{\pm 1}{\zeta_n}\sum_{i=1}^{2^L}Z_i$, and 
$b_n:=\zeta_n^2/n$. One has using independence, and the fact that $\frac{\zeta_n}{n}\cdot |Z_1|$ is bounded, 
\begin{equation*}
\bE[\exp(\theta b_n X_n]=
\Big(\bE[\exp(\theta \frac{\zeta_n}{n} Z_1] \Big)^{2^L}\\
= \Big(1+\frac{\theta^2}{2}\big(\frac{\zeta_n}{n}\big)^2\cdot
\bE[Z_1^2]+
\mathcal O\big(\big(\frac{\zeta_n}{n}\big)^3\cdot \bE[|Z_1|^3]\big) \Big)^{2^L}.
\end{equation*}
Note that $2^L\cdot \bE[Z_1^2]/n$ converges to $\sigma^2>0$, 
and that the fourth centered moment of $|\cR_n|$ 
is $\mathcal O(n^2 (\log n)^2)$, as proved by Le Gall in \cite{LG86}.
Thus, using that $\bE[|Z_1|^3]\le \bE[Z_1^4]^{3/4}$, we have
\[
\big(\frac{\zeta_n}{n}\big)^3\cdot \bE[|Z_1|^3]\le C
\big(\frac{\zeta_n\log n}{n}\big)^{3/2}.
\]
It follows that for any $\theta\in \R$,
\begin{equation}\label{GE-2}
\lim_{n\to\infty} \frac{n}{\zeta_n^2} \log \bE[
\exp\Big(\theta\frac{\zeta_n}{n} X_n\Big)=
\frac{\sigma^2}{2} \theta^2, 
\end{equation}
and one can then apply G\"artner--Ellis' Theorem, which concludes the proof of Theorem \ref{theo-gauss}.

\section*{Acknowledgments}
We thank Niccolo Torri for valuable help 
at an early stage of this project.
We acknowledge the support of the projects SWiWS (ANR-17-CE40-0032) and MALIN (ANR-16-CE93-0003).

\end{document}